\documentclass[11pt]{amsart}
\usepackage{geometry}                
\usepackage{graphicx}
\usepackage{amssymb}
\usepackage{epstopdf}
\renewcommand{\div}{\textrm{div}}
\newcommand{\curl}{\textrm{curl}}
\newcommand{\R}{\mathbb{R}}
\newcommand{\Id}{\textrm{Id}}
\newcommand{\T}{\mathbb{T}}
\newcommand{\Z}{\mathbb{Z}}
\newcommand{\eps}{\varepsilon}

\newcommand{\calA}{\mathcal{A}}
\newcommand{\calT}{\mathcal{T}}
\newcommand{\calV}{\mathcal{V}}

\newtheorem{thm}{Theorem}

\title{Volume enclosed by a flux surface}
\author{R.S.MacKay}
\address{Mathematics Institute, University of Warwick, Coventry CV4 7AL, UK}
\email{R.S.MacKay@warwick.ac.uk}
\date{\today}                                           

\begin{document}
\begin{abstract}
The paper describes ways that the computation of the volume enclosed by an invariant torus (flux surface) for a magnetic field can be reduced from a 3D integral to a 2D integral.
\end{abstract}

\maketitle
\section{Introduction}

In generality, a {\em magnetic field} is a vector field $B$ on an orientable 3D manifold $M$ that preserves a volume-form $\Omega$ (for the standard Euclidean volume, one writes $\div\, B = 0$).  Defining the associated flux 2-form $\beta = i_B\Omega$, the volume-preservation condition can be written as $d\beta = 0$ ($\beta$ is closed).  For a tutorial on use of differential forms in plasma physics, see \cite{M}.

A stronger condition (for general $M$) is that $\beta$ be exact, i.e.~$\beta = d\alpha$ for some 1-form $\alpha$.  It is usual to write this in terms of the associated vector potential $A$, related by $\alpha = A^\flat$ (where for a Riemannian metric $g$, $A^\flat$ is defined by $A^\flat(\xi) = g(A,\xi)$ for all tangent vectors $\xi$).  Exactness is equivalent to assuming $\int_S \beta = 0$ for any closed surface $S$ (not just those that bound a volume).  Most discussions of magnetic fields assume this stronger condition and we shall do so.

An invariant torus for $B$ is called a {\em flux surface}.  If $B$ is smooth enough (e.g.~$C^3$ and the third derivative is H\"older continuous), then various simple conditions imply a set of positive volume of flux surfaces, for example, existence of a generic elliptic closed fieldline or existence of one flux surface with smooth enough conjugacy to a Diophantine rotation (KAM theory, for a semi-popular introduction, see  \cite{KAMref}).

In the design of magnetic confinement devices for plasma, it is often desired to quantify the volume $V$ enclosed by a flux surface $S$.  In particular one might want to know the volume enclosed by the outermost flux surface of given class contained within the vacuum vessel, but we pose the question more generally.  Once a flux surface has been calculated, this can be done by 3D integration, but the question arises whether there is a more efficient way and that uses the invariance of $S$ under $B$.

One simple improvement is to reduce the problem to a surface integral:~if $M$ is contractible then $\Omega$ can be written as $d\nu$ for a 2-form $\nu$ (in many ways, e.g.~for $\Omega = dx \wedge dy \wedge dz$ one can take $\nu = x\, dy \wedge dz$ or cyclic permutations) and then $\int_V \Omega =  \int_S \nu$, by Stokes' theorem.  But this still does not use the invariance under $B$.

One way to use invariance under $B$ is the formula
\begin{equation}
V = \int_D T \beta,
\label{eq:Tbeta}
\end{equation}
where $D$ is any disk transverse to $B$ whose boundary is on the flux surface $S$, and $T$ is the first-return time to $D$ along the fieldline flow $\dot{x}=B(x)$.  The formula can be obtained from the identity $\Omega = \frac{B^\flat}{|B|^2} \wedge \beta$ by decomposing the domain into the bundle of fieldline segments from $D$ to $D$.  The same can be done for the volume between two nested tori by replacing $D$ by any annulus transverse to $B$ whose boundary components are on the respective tori; this includes the limiting case where one of the tori is just an elliptic closed fieldline, called a magnetic axis.  Although (\ref{eq:Tbeta}) is a 2D integral, it involves computing the return time function, which effectively makes the integration 3D, so there is no real saving.  

Here we give various formulae for the volume enclosed by a flux surface that use the invariance under $B$ and promise to be more efficient.
We build up from the simplest situation to the general one.

Such an important topic has, of course, been addressed before, e.g.~\cite{H} in the context of magnetohydrostatic (MHS) fields.  We will mention connections as we go along.

\section{Quasisymmetric fields}

We start by addressing the special case of a {\em quasisymmetric} magnetic field $B$, i.e.~there is a vector field $u$ independent from $B$ almost everywhere (in the region of interest), with $L_u\beta=0$, $L_u\Omega=0$ and $L_uB^\flat=0$.  Under mild additional conditions \cite{BKM}, which we assume, every orbit of $u$ is closed and has the same period $\tau>0$.

The simplest examples are the axisymmetric fields, those for which $u = \partial_\phi$ in cylindrical coordinates (then $\tau = 2\pi$).  It is an open question whether there are any other exactly quasisymmetric fields (compare Kovalevskaya's tops \cite{K}), but fields with non-axisymmetric quasisymmetry to a high degree of accuracy can be made \cite{LP}.

The conditions $L_u\beta=0$, $L_u\Omega=0$ imply that $i_u\beta$ is a closed 1-form.  
Assuming there is no homological obstruction then it is exact, i.e.~there exists a function $\psi$ such that $i_u\beta = d\psi$.
It follows that $i_ud\psi$ and $i_Bd\psi=0$, so $u$ and $B$ are tangent to the level sets of $\psi$.  The bounded regular components of level sets of $\psi$ are tori, because they are orientable surfaces supporting a nowhere-zero vector field.

Because both $B$ and $u$ are volume-preserving, $L_u\beta=0$ is equivalent to the commutator $[u,B]$ being $0$.
Then $(u,B)$ generate an action $\phi$ of $\R^2$ on the flux surfaces; for $t=(t_u,t_B) \in \R^2$ and $x$ on a flux surface $S$, $\phi_t(x)$ is the point reached by flowing for time $t_u$ with $u$ and $t_B$ with $B$ (the order does not matter because the fields commute).
We already have that $\phi_{(\tau,0)}$ is the identity.  Choose a $u$-line $\gamma$ and a point $x \in \gamma$.  There is a first $T>0$ such that the flow of $B$ starting at $x$ returns to $\gamma$.  $T$ is independent of $x \in \gamma$ because if $y \in \gamma$ then $y=\phi_{(t_u,0)}(x)$ for some $t_u$, so $\phi_{(0,T)}(y) = \phi_{(t_u,T)}(x) = \phi_{(t_u,0)}(\phi_{(0,T)}(x))$ is in $\gamma$ and is not for smaller $T$.  Furthermore, $T$ does not depend on the choice of $u$-line $\gamma$ on $S$ (but it does in general depend on the flux surface $S$).

Define $V$ to be the volume enclosed by a flux surface.  It is locally a function of $\psi$.  

\begin{thm}
$dV = \tau T(\psi) d\psi.$
\end{thm}

\begin{proof}
The magnetic flux across an annulus with boundary components $\gamma$ and $\gamma'$ on two neighbouring flux surfaces is $\tau d\psi$.  The volume it sweeps out when following $B$ is the integral of the return time with respect to the flux. 
\end{proof}

So one can compute the volume enclosed by a flux surface by integrating $\frac{dV}{d\psi}$ from some reference $\psi_0$ where the volume is known (e.g.~magnetic axis where it is zero), if one computes the return time function $T(\psi)$.  The result is a 2D integration, which is a saving over the original 3D integration.

\section{Fields with weak quasisymmetry}

A magnetic field has {\em weak quasisymmetry} $u$ if $L_u\beta=0$, $L_u\Omega=0$ and $L_u|B|=0$; this is a translation to differential forms of the conditions defined in \cite{RB}.  The latter is equivalent to $i_BL_u B^\flat=0$, so weak quasisymmetry is a generalisation of quasisymmetry.  The status of weak quasisymmetry as the zeroth order conditions for a velocity-dependent Hamiltonian symmetry of guiding-centre motion is discussed in \cite{BKM2}.

The preceding arguments go through with the exception that the period $\tau$ of the $u$-lines is now in general a function of $\psi$, so it needs computing too. 

\begin{thm}
$dV = \tau(\psi) T(\psi) d\psi.$
\label{thm:tauTdpsi}
\end{thm}

Thus, computing $V$ is now a 1D integration of a pair of 1D integrations, still equivalent to 2D.

\section{Fields with flux-form symmetry}
\label{sec:ffsymm}

Next we generalise to magnetic fields $B$ with a volume-preserving field $u$ independent from $B$ almost everywhere (in the region of interest), such that $L_u \beta = 0$ (dropping the requirement of $i_BL_uB^\flat=0$).  We call this a {\em flux-form symmetry}.  Equivalently, $u$ is a volume-preserving field that commutes with $B$.  

This enlarges the previous context to include, for example, any non-degenerate magnetohydrostatic (MHS) field $B$:~simply take $u$ to be the current density $J = \curl\, B$.  The MHS property $J\times B = \nabla p$ for pressure field $p$ implies that $L_J\beta = 0$.  Non-degeneracy means $dp \ne 0$ almost everywhere, which thereby makes $J$ and $B$ independent almost everywhere (it should be noted, however, that no non-axisymmetric non-degenerate MHS fields are known).

As before, assuming there is no homological obstruction then $i_u\beta=d\psi$ for some function $\psi$ (in the MHS case with $u=J$ one can take $\psi = p$), and the bounded regular components of level sets of $\psi$ are tori invariant under both $B$ and $u$.


Again, $(u,B)$ generate an action $\phi$ of $\R^2$ on the flux surfaces.
This implies that for each flux surface there is a lattice $\Gamma$ in $\R^2$ (discrete subset closed under subtraction) such that $\phi_t =\Id$ iff $t \in \Gamma$.  It can no longer be assumed that there is a $\tau>0$ with $(\tau,0) \in \Gamma$.  Let $T_1, T_2 \in \R^2$ be generators for $\Gamma$.  They can be computed by integration of ODEs for $u$ and $B$ and root finding.  Readers may recognise this as a strategy to prove the Liouville-Arnol'd theorem (e.g.~\cite{LAref}), which gives a parametrisation of $S$ by $\R^2/\Z^2$ with respect to which both $u$ and $B$ are constant.  In the magnetic field context, it was rediscovered by Hamada \cite{Ha}.

Make a matrix $\calT$ with the vectors $T_1,T_2$ as columns and let $\Delta = \det \calT$, which is a function of $\psi$.  


\begin{thm}
$dV = \Delta(\psi) d\psi.$
\label{thm:Deltadpsi}
\end{thm}

\begin{proof}
In a Liouville-Arnol'd parametrisation of a flux surface by $(\theta^1,\theta^2) \in \R^2/\Z^2$, $u$ and $B$ have constant components $(u^1,u^2), (B^1,B^2)$, so make a matrix $\calV$ with the components of $u$ and $B$ as columns. Then $d\psi = i_ui_B\Omega = \det \calV\, i_{\partial_{\theta_1}}i_{\partial_{\theta_2}} \Omega = \det \calV \, dV$.   But the lattice generators satisfy $\calT \calV = I$, the identity matrix, so $\det \calV \det \calT = 1$ and hence $dV = \det \calT d\psi$.
\end{proof}

From this one can again deduce the volume enclosed by a flux surface, by 1D integration with respect to $\psi$ from a reference case.  As the method requires a 1D set of 1D integrations (to compute $\Delta(\psi)$), it can again be considered a 2D integration overall, but that is still a saving compared to the 3D integration of $\Omega$.


Quasi-symmetric and weak quasisymmetric fields are special cases of this, with $[T_1\, T_2] = \left[\begin{array}{cc} \tau & c \\ 0 & T \end{array} \right]$ for some $c$ (the value such that $\phi_{(0,T)}(x) = \phi_{(-c,0)}(x)$).

\section{Fields with a foliation by flux surfaces}

If there is a foliation by flux surfaces (without knowing if there is a symmetry $u$) then one can still compute the enclosed volume by an essentially 2D integration, as follows.

For any invariant surface $S$ under $B$, the 2-form $\beta = dA^\flat$ vanishes on its tangent spaces.  It follows that $$\Phi = \int_\gamma A^\flat$$ is the same for all closed curves $\gamma$ on $S$ in the same homology class.  It can be interpreted as a magnetic flux.  Specifically, if $\gamma$ is the boundary of a disk $D$ then $\int_\gamma A^\flat = \int_D \beta$; if $\gamma - \gamma_0$ is the boundary of an annulus $A$ then $\int_\gamma A^\flat - \int_{\gamma_0} A^\flat = \int_A \beta$.  Think of the cases of a poloidal loop and of a toroidal loop relative to a magnetic axis $\gamma_0$.

We restrict attention to the case that $B$ is nowhere zero in the region of interest. Using volume-preservation and the foliation by invariant tori, one can deduce that the flow on each torus is of Poincar\'e type (in the terminology of \cite{BGKM}), i.e.~it has a transverse section such that every $B$-line crosses it in both directions.  Specifically, label the flux surfaces as level sets of a smooth function $\psi$ with $d\psi \ne 0$, choose any vector field $n$ such that $i_n d\psi = 1$ (e.g.~choose a Riemannian metric and let $n = \nabla \psi/|\nabla \psi|^2$), and let $\calA = i_n\Omega$. Then $\calA$ is non-degenerate on each flux surface and preserved by $B$.  This rules out the case of Reeb components (an annulus bounded by periodic orbits in opposite directions), because area would be contracted on approaching a bounding periodic orbit, so for a nowhere-zero field, the only case left is Poincar\'e type.

Given a closed curve $\gamma$ on a flux surface $S$, transverse to $B$, define $\bar{T}$ to be the limiting average return time to $\gamma$ of an infinitely long $B$-line starting at some point of $\gamma$.  By the theory of Poincar\'e flows on a torus, this average exists.  

Furthermore, if $B$ winds irrationally on $S$ it is the same for all initial conditions on $\gamma$.  In this case, $\gamma$ can be deformed to one for which the return time is constant, but for computational purposes it is enough to estimate the average for a given $\gamma$.  The weighted Birkhoff average method \cite{Meiss} is a promising route to do this accurately.  Alternatively, we propose approximating the integral of the return time from its values along one orbit segment (see section~\ref{sec:rettime}).

In the case of a rational surface, all trajectories are closed (this follows from preservation of $\calA$ on flux surfaces, specifically $\int_\eta i_B\calA$ from a reference point to an arbitrary point $x$ of $S$ along any curve $\eta$ on $S$ is preserved by $B$ if $\eta$ is chosen to depend continuously on $x$ \cite{N}) and we define $\bar{T}$ to be the average of the return time with respect to the area $\calA$.  But if the set of rational surfaces have measure zero then we will not really need this.

If the space is locally foliated by flux surfaces and $\gamma$ is chosen continuously then we obtain in the same way again
\begin{thm}
$dV = \bar{T} d\Phi.$
\label{thm:TdPhi}
\end{thm}
Note that in the case of a field with a flux-form symmetry, this result is consistent with that of Theorem~\ref{thm:Deltadpsi}, because given $u$, construct Liouville-Arnol'd coordinates $(\theta^1,\theta^2)$, suppose the second component $B^2 \ne 0$ (else interchange the coordinates), choose $\gamma$ to be $\theta^2=0$, then $d\Phi = i_{\partial_{\theta_1}}i_B\Omega = B^2 dV$ and the return time $T=1/B^2$ (constant).

The work required to use Theorem~\ref{thm:TdPhi} is slightly larger than 2D integration, in the sense that estimating the mean return time for a torus requires more than one revolution.  But we think it will still be less than a full 3D integration.

If the vector potential $A$ is not provided then computing it at a point from $B$ is a 1D integral (see the treatment of Poincar\'e's lemma in section~\ref{sec:gen}), so $\Phi$ is a 2D integration and then $V$ is a 3D integration.  But one can replace the computation of $\Phi$ by $d\Phi = \int_\gamma i_Y\beta$ where $Y$ is a vector field on $\gamma$ pointing to its homologue on an infinitesimally nearby flux surface, so use of $A$ was not really necessary, and this remains a 2D integration.

\subsection{Connections to established results}

Let us connect to results surveyed in \cite{H}.  Eq.~(19) of \cite{H} is
$$V'(\psi) = \int_0^{2\pi} d\alpha \int \frac{d\ell}{|B|},$$
where the $\ell$-integral is taken for one toroidal revolution, $\psi$ has been chosen to be the toroidal flux enclosed by the surface divided by $2\pi$, and $\alpha$ is a fieldline label such that $\beta = d\psi \wedge d\alpha$ (making $(\alpha,\psi)$ into ``Clebsch coordinates'' for $B$).  Now, $\int \frac{d\ell}{|B|}$ is the return time along fieldline flow to a poloidal section, the average with respect to $\calA$ is the average with respect to $\alpha$, and $\Phi = 2\pi \psi$.  So the formula agrees with that of Theorem~\ref{thm:TdPhi}.

Eq.~(32) of \cite{H} says that in non-degenerate MHS, the period of all fieldlines on a rational surface is the same (Hamada condition), so ``all flux tubes carrying the same magnetic flux on a given surface must have the same volume''.  Thus, $\bar{T}$ can be replaced by the period divided by the number of toroidal revolutions, which recovers the result in the equation before (22) in \cite{H}.  The Hamada condition is a particular consequence of the Liouville-Arnol'd theorem used in section~\ref{sec:ffsymm}.  A more general result for non-degenerate MHS, cited just after (30) in \cite{H}, is that the Jacobian of the Hamada coordinate system is a flux function.  Indeed, Hamada coordinates are ones in which $J$ and $B$ are constant on each flux surface and the Jacobian is $\Delta(\psi)$ from Theorem~\ref{thm:Deltadpsi}.

\subsection{Average return time from one orbit segment}
\label{sec:rettime}

Here is the proposed method to estimate the mean return time from one orbit segment (more accurately than just taking the average of the return time).  It applies to tori with irrational winding ratio $\iota$.  Then there is a parametrisation of $\gamma$ by $\theta \in \R/\Z$ such that the first return map to $\gamma$ is $\theta' = \theta+\iota$.  Express the return time as a function $T(\theta)$.  We desire its average $\bar{T}$ with respect to $\theta$,  i.e.~$\bar{T} = \int T(\theta)\, d\theta$.
Given an orbit segment with successive return times $T_n, n = 1,\ldots, N$, we can label the initial point by $\theta=0$ and then $T_n = \tau((n-1)\iota)$.  The values $(n-1)\iota$ can be reduced to the interval $[0,1)$ and $\int \tau(\theta) \, d\theta$ approximated by the trapezoidal rule applied to these points.

The method requires to first estimate $\iota$.  This can be done by choosing some coordinate $\phi: \gamma \to \R/\Z$, with initial point $\phi=0$, and looking at the order in which the successive returns $\phi_n$ come.  Let $n_k$, $k\ge 1$, be the indices for the successively closest returns to the initial point (they will be on alternate sides of $0$) and $m_k$ be the closest integer to the number of revolutions of $\gamma$ made (this requires choosing the same convention as for $\iota$).  Then $\iota$ can be estimated from the linear approximation to $\phi(\theta)$ near $\theta=0$:
$$\frac{n_{k}\iota - m_{k}}{n_{k-1}\iota - m_{k-1}} \sim \frac{\phi_{n_{k}}}{\phi_{n_{k-1}}}.$$
So
$$\iota \sim \frac{m_k \phi_{n_{k-1}}-m_{k-1}\phi_{n_k}}{n_k \phi_{n_{k-1}}-n_{k-1}\phi_{n_k}}.$$

Note that there will be a sequence $a_k$ of positive integers such that $n_{k+1} = a_k n_k + n_{k-1}$ and the same for $m_{k+1}$ (defining $n_0 = 1, m_0 = 0$).  This sequence gives the continued fraction expansion of $\iota$.

\section{General fields}
\label{sec:gen}

For a general magnetic field there may be a set of flux surfaces of given class, interspersed with island chains and ``chaos'' (the class can be defined relative to a magnetic axis, for example).  

Can one compute the integral $\int_{D'} T \beta$ of the return time $T$ over the part $D'$ of a transverse section $D$ that is inside a flux surface in a more efficient way than the obvious 2D integral of the result of 1D integrations?

Yes.  $T\beta$ is a top-form on $D$.  If $D$ is contractible (e.g.~a poloidal section) then choose a vector field $X$ on $D$ whose flow $\phi$ contracts $D$ to a point as $t \to +\infty$, and for any tangent $\xi$ to $D$, let $\eta(\xi) = -\int_0^\infty T\beta(\phi_{t*}X,\phi_{t*}\xi)\, dt$ (this is one way to prove Poincar\'e's lemma).  Then $T\beta = d\eta$ on $D$ and so for any subset $D' \subset D$, $\int_{D'} T\beta = \int_{\partial D'} \eta$.

To make it concrete, suppose coordinates $(u,v)$ on $D$ such that the forward flow of $X(u,v) = (-u,-v)$ is defined on $D$.  It is just $\phi_t(u,v) = e^{-t}(u,v)$.  Then for a tangent vector $\xi$ at $r=(u,v)$, $\eta(\xi) = \int_0^\infty e^{-2t} T \beta(r,\xi)\, dt$, where $T$ and $\beta$ are evaluated at $e^{-t}r$.  Recall that $\beta = i_B\Omega$, so $\beta(r,\xi) = \Omega(B,r,\xi)$, the triple product in the Euclidean case.

Thus, we have reduced the calculation of the volume enclosed by a flux surface to a 1D integral (around $\partial D'$) of a 1D integral $\eta(\xi)$ with $\xi$ the tangent to $\partial D'$.  This is essentially a 2D integral.

If $D$ is not contractible, e.g.~an annulus between a closed curve on $S$ and a magnetic axis $\gamma_0$, then can instead contract to $\gamma_0$ and obtain a similar reduction to integration of a 1-form, that we do not spell out here.

But the question remains whether this is any better than integrating a primitive of $\Omega$ over the surface.


\section{Fluxes}
It is interesting to compare computation of enclosed volume with the question of computing the fluxes of a flux surface.

Given a ``poloidal'' closed curve $\gamma$ on a flux surface $S$ bounding a disk $D$ in $M$, its flux is $\Phi = \int_D \beta$.  This a 2D integral.
It can be written as $\Phi = \int_\gamma A^\flat$, reducing the computation to a 1D integral (assuming that a vector potential $A$ is given, else its computation is already a 1D integration and no saving results).
By exactness of $\beta$, $\Phi$ is independent of the disk $D$ chosen to span $\gamma$.  By invariance of the flux surface, it is also independent of the representative $\gamma$ of its homology class in $H_1(S)$.  This is because $\beta = 0$ on all pairs of tangent vectors to $S$.

To define the flux for a general homology class on a flux surface requires some more discussion.  For a flux surface $S$ bounding a solid torus with $B$ nowhere-zero inside, there is a closed fieldline inside.  Choosing one, designated the magnetic axis $\gamma_0$, one can define the flux for a non-poloidal closed curve $\gamma$ on $S$ to be $\int_A \beta$ where $A$ is the annulus with boundary $\gamma - \gamma_0$.  Again, this can be reduced to 1D integrals $\int_\gamma A^\flat - \int_{\gamma_0} A^\flat$.

As already mentioned, for a field with flux-form symmetry $u$ with closed $u$-lines, $d\Phi = \tau d\psi$, where $\tau(\psi)$ is the period of the $u$-lines. 

\subsection{Percival's variational principle}

The fluxes are associated with one strategy for computing flux surfaces, namely an adaptation of Percival's variational principle for invariant tori of Hamiltonian systems \cite{P} to the magnetic field context (see \cite{GKN} for another extension of Percival's variational principle).

Let $A$ be a vector potential for $B$.  Given $\omega \in \R^2 \setminus \{0\}$ and a differentiable mapping $x: \T^2 = \R^2/\Z^2 \to M$, define $P$ (for Percival) by
$$P_\omega(x) = \int A_j(x(\theta))  \frac{\partial x^j}{\partial \theta^i}(\theta)\omega^i  \, d^2\theta.$$
The first variation of $P_\omega(x)$ with respect to a variation $\delta x$ is
$$\delta P_\omega = \int   A_{j,k}\delta x^k x^j_{,i} \omega^i + A_j \delta x^j_{,i}\omega^i \, d^2\theta,$$
where subscript $_{,i}$ denotes partial derivative in direction $i$.
The second term can be integrated by parts to produce
$$\delta P_\omega = \int  (A_{j,k}-A_{k,j})x^j_{,i} \omega^i \delta x^k\, d^2\theta = \int  \eps_{jkl}B^lx^j_{,i} \omega^i \delta x^k\, d^2\theta,$$
where $\eps$ is the Levi-Civita symbol.
This is zero for all variations $\delta x$ iff $x^j_{,i}\omega^i $ is parallel to $B^j$, i.e.~there is a function $c: \T^2 \to \R$ such that 
$$x^j_{,i}(\theta)\omega^i = c(\theta) B^j(x(\theta)).$$  In particular, if $x$ is a diffeomorphism then the image of $x$ is a flux surface (furthermore, $c$ is nowhere zero).

In more detail, this says that the function $x$ conjugates $\dot{\theta} = \omega/c(\theta)$ to $\dot{x}=B(x)$ on the flux surface.  Furthermore, assuming $B$ is nowhere zero on the surface, the fieldlines on the surface have winding ratio $\iota = \omega^2/\omega^1$ with respect to $(\theta^1,\theta^2)$.  The $\theta$ motion has invariant probability density $c(\theta)/\bar{c}$, where $\bar{c}$ is the average of $c$ with respect to $d^2\theta$.

For irrational winding ratio, we can use ergodicity.  $P_\omega(x)$ is the $\theta$-average of $A\cdot \frac{dx}{d\tau}$ for $\frac{d\theta}{d\tau}=\omega$.  Transforming time by $dt = c(\theta) d\tau$, we see that $P/\bar{c}$ is the average of the helicity $A\cdot B$ with respect to the invariant probability on $S$.  By unique ergodicity of the $\theta$ motion, this is the $t$-average of $A\cdot B$ along any fieldline on $S$.
Similarly, $\bar{c}$ has the interpretation that $\omega_j/\bar{c}$ is the time-averaged numbers of revolutions $\Delta \theta_j$ along a fieldline.

Note that both $P$ and $c$ are homogeneous of first degree in $\omega$.  So one can consider $P/\bar{c}$ to be a function of $\iota$.

In the differentiable context, there is no reason for $P$ to have a critical point.
A main point of Percival's variational principle is that it allows extension to functions $x$ that are differentiable in direction $\omega$ but not necessarily continuous.  Indeed, Mather used the principle to prove existence of a set $M_\rho$ of quasiperiodic orbits for each irrational rotation number $\rho$ for area-preserving twist maps \cite{Ma}.  $M_\rho$ is either a circle or a Cantor set, thus providing an answer in this context to the question of what happens to invariant tori beyond where KAM theory applies.  The case of a Cantor set was christened ``cantorus'' by Percival \cite{P2}.

One can expect the same conclusion to apply to magnetic fields as long as they have shear, meaning that in some coordinate system $(r,\phi_1,\phi_2)$, $\frac{\partial }{\partial r}\frac{B^{\phi_2}}{B^{\phi_1}} \ne 0$.

One consequence of the variational principle is that if one has computed the stationary value of $P_\omega$ for each $\omega$ (it is enough to do for each $\iota$ by homogeneity in $\omega$) then one can read off the toroidal and poloidal fluxes $\Phi_j$ as follows.  By invariance of the flux surface, the integral of $A^\flat$ along a long fieldline can be deformed into the weighted sum of its integrals around a basis for cycles on the flux surface, $P_\omega= \Phi_1 \omega_1 + \Phi_2 \omega_2$.
$P_\omega$ is stationary with respect to variations in $x$, so assuming differentiability of the critical value of $P_\omega$ with respect to $\omega$ (which could probably be proved along the lines of \cite{Ma2}) we deduce that $\frac{\partial P}{\partial \omega_j} = \Phi_j$.

This formula is an analogue of that for the area under an invariant circle or cantorus for area-preserving twist maps as derivative of mean action with respect to rotation number \cite{Ch}, and the equivalent one for chemical potential of an incommensurate structure for Frenkel-Kontorova chains as derivative of the mean energy with respect to the mean spacing \cite{Au}.  

\section{Computing flux surfaces}
The paper begs the question of how to compute flux surfaces.  This might depend on the context, but one way is Percival's variational principle.  It can be formulated as looking for a critical point of a functional in a space of double periodic functions, which could be represented by Fourier series.  One may not know in advance, however, whether there exists a smooth solution.  If the field has shear then one can do a partial stationarity to replace a minor radius coordinate by first derivatives with respect to the others, rendering a functional that is bounded below and hence amenable to minimisation methods.  But a minimiser might still not be smooth.

Another method is to look for flux surfaces as limits of sequences of periodic orbits.  Again, this might not give a flux surface, but if the ``residues'' of the periodic orbits go to zero, they are quite likely to accumulate on a flux surface.  A formal result is available in the other direction (if there is a smooth flux surface with shear, then there are sequences of periodic orbits that converge to it with residues going to zero \cite{M2}).  Outermost flux surfaces can be found as boundaries of limits of sequences of periodic orbits whose residues are bounded \cite{GMS}.

\section{Conclusion}
The paper has presented methods to reduce the computation of the volume enclosed by a flux surface of a magnetic field from 3D to 2D integration.

A question remains:~are any of the methods more efficient than the first one that does not use invariance of the flux surface under the field?  The answer is likely to depend on how the field is presented (e.g.~directly or via a vector potential) and how the flux surfaces are presented (e.g.~as a graph of minor radius against two angle variables, as the closure of the orbit of one point...).

It would be interesting to test the methods numerically.

\section*{Acknowledgements}
This work was supported by a grant from the Simons Foundation (601970, RSM).  I am grateful to David Martinez for detailed comments.

\end{document}